\documentclass[a4paper]{amsart}
\title{$A_k$ singlarities of wave fronts.} 
\date{October 25, 2008}
\usepackage{amsthm}
\theoremstyle{plain}
 \newtheorem{theorem}{Theorem}[section]
 \newtheorem{proposition}[theorem]{Proposition}
 \newtheorem{lemma}[theorem]{Lemma}
 \newtheorem{corollary}[theorem]{Corollary}
\theoremstyle{remark}
 
 \newtheorem{remark}[theorem]{Remark}
 \newtheorem*{remark*}{Remark}
 \newtheorem*{acknowledgements}{Acknowledgements}
 \newtheorem{example}[theorem]{Example}
\numberwithin{equation}{section}
\numberwithin{figure}{section}
\author{Kentaro Saji}
\address[Saji]{%
  Department of Mathematics,
  Faculty of Educaton,
  Gifu University,  Yanagido 1-1, Gifu 151-1193, Japan
}
\email{ksaji@gifu-u.ac.jp}

\author{Masaaki Umehara}
\address[Umehara]{%
   Department of Mathematics, Graduate School of Science,
   Osaka University,
   Toyonaka, Osaka 560-0043,
   Japan
}
\email{umehara@math.sci.osaka-u.ac.jp}
\author{Kotaro Yamada}
\address[Yamada]{%
   Faculty of Mathematics,
   Kyushu University,
   Higashi-ku, Fukuoka 812-8581, Japan%
}
\email{kotaro@math.kyushu-u.ac.jp}
\thanks{
 Kentaro Saji was supported by the
 Grant-in-Aid for Young Scientists
 (Start-up) No.~19840001, from the Japan Society for the Promotion of 
 Science.
 Masaaki Umehara and Kotaro Yamada were supported by 
 the Grant-in-Aid for 
 Scientific Research (A) No.~19204005 
 and Scientific Research (B) No.~14340024,
 respectively, from the Japan Society for the Promotion of Science.
}
\dedicatory{Dedicated to Professor Yoshiaki Maeda on
the occasion of his sixtieth birthday.}
\usepackage[dvips]{graphicx} 
\usepackage{verbatim,enumerate}
\usepackage[usenames]{color}
\usepackage{amsmath,amssymb}
\newcommand{\R}{\boldsymbol{R}}
\newcommand{\C}{\boldsymbol{C}}
\newcommand{\K}{\boldsymbol{K}}
\newcommand{\Z}{\boldsymbol{Z}}
\newcommand{\vect}[1]{\boldsymbol{{#1}}}
\newcommand{\nuin}{\nu_{\operatorname{in}}}
\renewcommand{\phi}{\varphi}
\newcommand{\inner}[2]{%
   \left\langle{#1},{#2}\right\rangle^{}_{\mbox{\!\tiny $\K$}}}
\newcommand{\dimK}{\dim_{\K}}
\newcommand{\spann}[1]{%
  \operatorname{Span}_{\mbox{\!\tiny $\K$}}\!\!\left\{#1\right\}}
\newcommand{\Image}{\operatorname{Im}}

\renewcommand{\theenumi}{{\rm(\arabic{enumi})}}
\renewcommand{\labelenumi}{\theenumi\ }
\begin{document}
\maketitle
\begin{abstract}
In this paper, we discuss the recognition problem for $A_{k}$-type
singularities on wave fronts.
We give computable and simple criteria of these singularities, 
which will play a fundamental role in generalizing 
the authors' previous
work ``the geometry of fronts'' for surfaces.
The crucial point to prove our criteria for $A_k$-singularities
is to introduce a suitable parametrization of the singularities  called 
the ``$k$-th KRSUY-coordinates'' (see Section~\ref{sec:adopted}).
Using them, we can directly construct a versal unfolding for a given
singularity.
As an application,
we prove  that a given nondegenerate singular point $p$ on a real 
(resp.\ complex) hypersurface (as a wave front) in $\R^{n+1}$ 
(resp.\ $\C^{n+1}$) is differentiably 
(resp.\ holomorphically)
right-left equivalent to the  $A_{k+1}$-type singular point
if and only if the linear projection of the singular set around $p$ into
a generic hyperplane $\R^n$ (resp.\ $\C^{n}$)
is right-left equivalent to the $A_{k}$-type singular point in 
$\R^{n}$ (resp.\ $\C^{n}$).
Moreover, we show that the restriction of a
$C^\infty$-map $f:\R^n\to \R^n$ to its Morin singular
set gives a wave front consisting of only $A_k$-type singularities.
Furthermore, we shall give a relationship
between the normal curvature map and the zig-zag numbers 
(the Maslov indices) of wave fronts.
\end{abstract}

\section{Introduction}
Throughout this paper, we denote by $\K$ the real number field $\R$ or
the complex number field $\C$.
Let $m$ and $n$ be positive integers.
A map $f\colon{}\K^m\to \K^n$ is called {\em $\K$-differentiable\/}
if it is a $C^\infty$-map when $\K=\R$, and is a holomorphic map when
$\K=\C$.
Let $\inner{~}{~}$ be the {\em $\K$-inner product\/} given by
\[
    \inner{X}{Y}=\sum_{j=1}^n x_j y_j
    \qquad
    \bigl(X=(x_1,\dots,x_n),~Y=(y_1,\dots,y_n)\in\K^n\bigr).
\]
Let $P^{n}(\K)$ be the $\K$-projective space and
$\K^{n+1}\setminus\{\vect 0\}\ni X \mapsto [X]\in P^{n}(\K)$
the canonical projection.
By using the above $\K$-inner product,
the $\K$-projective cotangent bundle $PT^*\K^{n+1}$
has the following identification
\[
    PT^*\K^{n+1}=\K^{n+1}\times P^{n}(\K),
\]
which has the canonical $\K$-contact structure.
Let $U\subset \K^{n}$ be a domain and 
\[
    L:=(f,[\nu]):U\longrightarrow \K^{n+1}\times P^{n}(\K)
\]
a Legendrian immersion, where 
$\nu$ is
a locally defined $\K$-differentiable map
into $\K^{n+1}\setminus\{\vect 0\}$ such that
\[
    \inner{df(\vect{u})}{\nu}=0\qquad (\vect{u}\in TU).
\]

In this situation, $f$ is called a {\em wave front\/} or {\em front},
and $\nu$ is called the {\em $\K$-normal vector field\/} of $f$.
If $\K=\C$,  $\inner{\nu}{\nu}$ 
might vanish.
A point $p\in U$ is called a {\em singular point\/} if the front $f$ is
not an immersion at $p$.

In this paper, we shall discuss the recognition problem for
$A_{k+1}$-type singularities on wave fronts.
These are fundamental singularities on wave fronts (see \cite{AGV}).
We give a simple and computable necessary and sufficient condition that
a given singular point $p$ on a hypersurface (as a wave front) in 
$\K^{n+1}$ is $\K$-right-left equivalent to the $A_{k+1}$-type singular 
point 
($k\leq n$; Theorem~\ref{thm:criteria}, 
Corollaries \ref{cor:actual}, \ref{cor:submanifold1} and
\ref{cor:submanifold2} 
in Section~\ref{sec:criteria}),
where two $\K$-differentiable map germs 
$f_j\colon{}(U_j,p_j)\to(\K^{n+1},p'_j)$ $(j=1,2)$ are
{\em $\K$-right-left equivalent\/} if there exist $\K$-diffeomorphism
germs 
$\psi\colon{}(U_1,p_1)\to (U_2,p_2)$ and 
$\Psi:(\K^{n+1},p'_1)\to(\K^{n+1},p'_2)$
such that
$\Psi\circ f_1=f_2\circ\psi$ holds.
Here, the $A_{k+1}$-type  singularity 
(or $A_{k+1}$-front singularity)
is a map germ defined by
\begin{equation}\label{eq:ak-def}
    X
      \longmapsto
    \left(
        (k+1)t^{k+2}+\sum_{j=2}^k(j-1)t^jx_j,
        -(k+2)t^{k+1}-\sum_{j=2}^kjt^{j-1}x_j,
        X_{1}
    \right)
\end{equation}
at the origin, where $X=(t,x_2,\dots,x_n)$, $X_{1}=(x_2,\dots,x_n)$.
The image of it coincides with the discriminant set 
$\{F=F_t=0\}\subset (\K^{n+1};u_0,\dots,u_n)$ 
of
the versal unfolding 
\begin{equation}\label{eq:versal}
    F(t,u_0,\dots,u_n):=t^{k+2}+u_{k} t^k+\dots +u_1 t +u_0.
\end{equation}
By definition, $A_1$-front singularities are regular points.
A $3/2$-cusp in a plane is an $A_2$-front singularity
and a swallowtail in $\R^3$ is an $A_3$-front singularity.

When $n=2$, useful criteria for cuspidal edges and swallowtails
are given in \cite{krsuy}. We shall give a generalization of
the criteria here.
The crucial point to prove our criteria for 
$A_{k+1}$-singularities
is to introduce  the ``$k$-th KRSUY-coordinates''
as a generalization of the coordinates for
cuspidal edges and swallowtails in \cite{krsuy}.
Using them, we can directly construct a versal unfolding
whose discriminant set coincides with the given singularity.  

As an application, when $\K=\R$, we show that
the restriction of a $C^\infty$-map $f:\R^{n+1}\to \R^{n+1}$ 
into its Morin singular set (see the appendix)
gives a wave front consisting of only $A_{k+1}$-type singularities 
$(k\leq n)$.
Moreover, in the final section,
we shall give a relationship
between the normal curvature map and the zig-zag numbers 
(the Maslov indices) of wave fronts.
\section{Criteria for $A_{k+1}$-front singularities}
\label{sec:criteria}

Let $U\subset (\K^n;x_1\dots,x_n)$ be a domain and
$f:U\to \K^{n+1}$ a front.
Since we will work on a local theory, we may assume that
the $\K$-normal vector field $\nu$ of $f$ is defined on $U$.
We define a $\K$-differentiable
function $\lambda$ on $U$ as the determinant
\begin{equation}\label{eq:lambda}
    \lambda:=\det(f_{x_1}^{} ,\ldots,f_{x_n}^{},\nu),
\end{equation}
where $f_{x_j}^{}:=\partial f/\partial x_j$
($j=1,2,\dots,n$). 
A point $p\in U$ is {\em $1$-singular\/}
(or {\em singular})
if $\lambda(p)=0$,  that is, $p$ is a singular point.
A point $p$ is {\em $1$-nondegenerate\/}
(or {\em non-degenerate}) if $p$ is $1$-singular and the exterior
derivative $d\lambda$ does not vanish at $p$.
The following assertion is obvious:
\begin{lemma}\label{lem:non-deg-indep}
 The definition of $1$-nondegeneracy is independent
 of the choice of local coordinate system $(x_1,\dots,x_n)$
 and the choice of a $\K$-normal vector field $\nu$.
 Moreover, if $p$ is $1$-nondegenerate, 
 the $\K$-linear map $df_p:T_pU\to \K^{n+1}$
 has a kernel of  $\K$-dimension exactly one.
\end{lemma}

If $p$ is 1-nondegenerate, the singular set of $f(=:f_0)$ 
\[
    S_1:=S(f)=\{q\in U\,;\, \lambda(q)=0\}
\]
is an embedded {\em $\K$-differentiable\/} hypersurface of $U$ near $p$.
We denote by $TS_1$ the $\K$-differentiable tangent bundle of 
$\K$-differentiable manifold $S_1$.
Then 
\[
    f_1:=f|_{S_1}:S_1\longrightarrow \K^{n+1}
\]
is a $\K$-differentiable  map.
Since $\dimK\ker(df_p)=1$,
we can take a sufficiently small neighborhood $V$ ($\subset U$)
of $p$ and 
non-zero $\K$-differentiable vector field $\eta$ on $S_1\cap V$ 
which belongs to the kernel of $df$, 
that is $df_q(\eta_q)=0$ for $q\in S_1\cap V$.
We call $\eta$ a {\em null vector field\/} of $f$.
Moreover, we can construct a $\K$-differentiable vector field 
$\tilde \eta$ on $V$ 
(called an {\em extended null vector field})
whose restriction
$\tilde \eta|_{S_1\cap V}$ on the singular set $S_1\cap V$ 
gives a null vector field.

\begin{remark}\label{rem:null}
 Since $\eta\in \ker(df)$ on $S_1$, we can write 
 the components of $\eta$ explicitly using
 determinants of $(n-1)$-submatrices of the Jacobian matrix of $f$.
 Then this explicit expression of $\eta$ 
 gives a $\K$-differentiable vector field 
 on a sufficiently small neighborhood of a singular point. 
 Thus, we get an explicit procedure
 to construct $\tilde \eta$.
 For example,
 let $f$ be a front
 $f=(f^1,f^2,f^3,f^4)\colon{}\K^3\to\K^4$
 such that $f^1_{x_1}f^2_{x_2}-f^1_{x_2}f^2_{x_1}\ne 0$.
 Then
 \[
    \tilde \eta=
    \left(
    \det\begin{pmatrix}
          f^1_{x_2} & f^1_{x_3} \\ 
          f^2_{x_2} & f^2_{x_3}
	\end{pmatrix},
   -\det\begin{pmatrix}
          f^1_{x_1} & f^1_{x_3} \\ 
          f^2_{x_1} & f^2_{x_3}
	\end{pmatrix},
    \det\begin{pmatrix}
          f^1_{x_1} & f^1_{x_2} \\
          f^2_{x_1} & f^2_{x_2}
	\end{pmatrix}
    \right)
 \]
 is an extended null vector field.
\end{remark}
Let $\lambda'\colon{}U\to\K$ and $f',f'':U\to\K^{n+1}$
be the derivatives $d\lambda(\tilde \eta)$,
$df(\tilde \eta)$ and $df'(\tilde \eta)$, respectively.
Then by definition, we have
\begin{equation}\label{eq:S1}
    S_1:=\{q\in U\,;\, f'(q)=\vect{0}\},
\end{equation}
where $\vect{0}=(0,\dots,0)$.
The following assertion holds:

\begin{lemma}\label{lem:zero}
 The singular set $S_2:=S(f_1)$ of $f_1\colon{}S_1\to \K^{n+1}$
 has the following expressions{\rm :}
 \begin{align*}
   S_2&= \{q\in S_1\,;\, \eta_q\in T_qS_1\}
       = \{q\in S_1\,;\, \lambda'(q)=0\}
       = \{q\in U\,;\, \lambda(q)=\lambda'(q)=0\}\\
      &= \{q\in S_1\,;\, f''(q)=\vect{0}\}
       = \{q\in U\,;\, f'(q)=f''(q)=\vect{0}\}.
 \end{align*}
\end{lemma}
\begin{proof}
 A point $q\in S_1$ belongs to $S_2$ if and only if there exists 
 $\xi\in T_qS_1\setminus\{\vect{0}\}$ such that
 $df_1(\xi)=0$.
 Since $\dimK\ker(df_q)=1$, the vector $\xi$ must be proportional to
 $\eta_q$.
 Thus we get
 $S_2=\{q\in S_1\,;\, \eta_q\in T_qS_1\}$.
 Since $S_1$ is a level set of $\lambda$,
 $\eta_q\in T_qS_1$ if and only 
 if $d\lambda(\eta_q)=\lambda'(q)=0$, 
 which proves the first equation. 
 By \eqref{eq:S1} and by
 $\eta\in TS_1$ on $S_2$, the derivative $f''=df'(\eta)$
 of $f'$ with respect to $\eta$ vanishes on $S_2$, 
 that is
 $S_2\subset \{q\in S_1\,;\, f''(q)=\vect{0}\}$. 
 Next we suppose $f''(q)=\vect 0$ ($q\in S_1$).
 Since $\lambda'=0$ on $S_2$,
 by taking a new coordinate system $(z_1,\dots,z_n)$ 
 on $U$ such that
 $\partial/\partial z_n=\tilde \eta$,
 we have 
 \[
   \lambda'(q)=\{\phi\det(f_{z_1},\dots,f_{z_{n-1}},f',\nu)\}'(q)
              =\phi(q)\det(f_{z_1},\dots,f_{z_{n-1}},f'',\nu)(q)
              =0,
 \]
 where
 $\phi:=\det(\partial z_i/\partial x_j)$ is the Jacobian.
 Consequently $q\in S_{2}$ holds,
 that is $S_2= \{q\in S_1\,;\, f''(q)=\vect{0}\}$.
\end{proof}

A $1$-nondegenerate singular point $p\in S_1$ is called 
{\em $2$-singular\/} if $p\in S_2$. 
Moreover, a $2$-singular point $p\in S_2$ is 
{\em $2$-nondegenerate\/} if $(d\lambda')_p\ne 0$ on $T_pS_1$.
In this case, $S_2$ is an embedded hypersurface of $S_1$ around  $p$ 
if it is $2$-nondegenerate.

Now, we define the $j$-singularity and the $j$-nondegeneracy
inductively ($2\leq j\leq n$):
Firstly, we give the following notations:
\begin{alignat*}{3}
    \lambda^{(0)}&:=\lambda,\qquad 
   &\lambda^{(1)}&:=\lambda',\qquad 
   &\lambda^{(l)}&:=d\lambda^{(l-1)}(\tilde \eta), \\
    f^{(0)}      &:=f,\qquad 
   &f^{(1)}      &:=f',\qquad 
   &f^{(l)}      &:=df^{(l-1)}(\tilde \eta),
\end{alignat*}
where $l=1,2,3,\dots$.
We fix a $(j-1)$-nondegenerate singular point $p$ of $f$.
Then the  $(j-1)$-st singular set 
\[
    S_{j-1}:=S(f_{j-2})
\]
is an embedded hypersurface of $S_{j-2}$ around $p$.
(Here we replace $U$ by a sufficiently small neighborhood of $p$
at each induction step if necessary. 
In particular, we may assume $S_0:=U$.)
We set
\[
       f_{j-1}:=f|_{S_{j-1}}:S_{j-1}\longrightarrow \K^{n+1}.
\]
As in the proof of  Lemma~\ref{lem:zero}, by an inductive use of
the identity
\begin{align*}
    \lambda^{(j-1)}(q)
   &=
    \{\phi\det(f_{z_1},\dots,f_{z_{n-1}},f',\nu)\}^{(j-1)}(q)
   \\
    &=
    \phi(q)\det(f_{z_1},\dots,f_{z_{n-1}},f^{(j)},\nu)(q),
\end{align*}
$S_j:=S(f_{j-1})$ satisfies
\begin{align}\label{eq:lambda-eta}
    S_j &=
           \{q\in S_{j-1}\,;\, \eta_q\in T_qS_{j-1}\}
    \\
        &=
           \{q\in S_{j-1}\,;\, \lambda^{(j-1)}(q)=0\}
         = \{q\in U\,;\, \lambda(q)=\dots =\lambda^{(j-1)}(q)=0\} 
    \nonumber \\
        &=
           \{q\in S_{j-1}\,;\,  f^{(j)}(q)=\vect{0}\}
         = \{q\in U\,;\, f'(q)=\dots=f^{(j)}(q)=\vect{0}\}.
    \nonumber
\end{align}
Each point of $S_{j}$ is called a {\it $j$-singular point}.
A point $q\in S_j$ is called {\em $j$-nondegenerate\/} if 
$(d\lambda^{(j-1)})_q\ne 0$ on $T_qS_{j-1}$, 
that is 
\begin{equation}\label{eq:not}
 T_qS_{j-1}\not \subset \ker(d\lambda^{(j-1)})_q.
\end{equation}
These definitions do not depend on the choice of a coordinate
system $(x_1,\dots,x_n)$ of $U$.
Under these notations, a criterion for $A_{k+1}$-front singularities
is  stated as follows:
\begin{theorem}\label{thm:criteria}
 Let $f:(U,p)\to(\K^{n+1},f(p))$  be a front, 
 where $U$ is a domain in $\K^n$.
 Then $f$ at\/ $p$ is $\K$-right-left equivalent to the
 $A_{k+1}$-front singularity if and only if $p$ is
 $k$-nondegenerate but is not $(k+1)$-singular $(k\leq n)$. 
\end{theorem}
Though $k$-singular points are defined only for $k\le n$,
we define any points are not $(n+1)$-singular.
This theorem is proved in 
Sections~\ref{sec:adopted} and \ref{sec:versal}.

\begin{corollary}\label{cor:actual}
 Let $f:U\to\K^{n+1}$ be a front.
 Then $f$ at\/ $p\in U$ is $\K$-right-left equivalent to 
 the $A_{k+1}$-front singularity $(1\leq k\leq n)$  
 if and only if
 \begin{equation}\label{eq:actual-1}
     \lambda=\lambda'=\cdots=\lambda^{(k-1)}=0,
           \quad  \lambda^{(k)}\ne 0
 \end{equation}
 at $p$ and the Jacobian matrix of $\K$-differentiable map 
 $\Lambda:=(\lambda,\lambda',\ldots,\lambda^{(k-1)}):U\to \K^{k}$
 is of rank $k$ at $p$.
\end{corollary}

\begin{corollary}\label{cor:submanifold1}
 Let $U$ be a domain in $\K^n$, 
 $f\colon{}U\to\K^{n+1}$ a front,
 and $p\in U$ a nondegenerate singular point.
 Take a local tangent frame field $\{v_1,\dots,v_{n-1}\}$
 of the singular set $S(f)$ {\rm(}a smooth hypersurface in $U${\rm)} 
 around $p$ and set
\begin{equation}\label{eq:mu}
     \mu:= \det(v_1,\dots,v_{n-1},\eta),
\end{equation}
 as the determinant function on $\K^n$
 where $\eta$ is a null vector field.
 Then $p$ is $2$-singular if and only if $\mu(p)=0$.
 In particular,
 $p$ is $\K$-right-left equivalent to
 the $A_{2}$-singularity if and only if $\mu(p)\ne 0$.
 Moreover, a $2$-singular point
 $p$ is $2$-nondegenerate if and only if $d\mu(T_p S(f))\ne \{0\}$.
 Furthermore, $p$
 is $\K$-right-left equivalent to
 the $A_{3}$-singularity
 if and only if
 $\mu(p)=0$ and $d\mu(\eta_p)\ne 0$
 hold.
 {\rm (}If $\mu(p)=0$, then the null vector $\eta_p$ at $p$
 is tangent to $S(f)$ and $d\mu(\eta_p)$ is well-defined.{\rm)}
\end{corollary}

\begin{remark}
 These assertions for $n=2$ 
 have been already proved in \cite{krsuy}.
 The criterion for
 $A_2$-singularities for general $n$ has been shown in \cite{suy}.
\end{remark}

Let $f\colon{}U\to\K^{n+1}$ a front, 
and $p\in U$ a $2$-nondegenerate singular point.
Then $S_2=S(f_1)$ is a submanifold of codimension $2$ in $U$ and 
Lemma~\ref{lem:zero} yields that the null vector field $\eta$
is a tangent vector field of $S_2$, and
we can set
\[
  \mu'=d\mu(\eta),\quad
  \mu''=d\mu'(\eta),\quad\cdots,\quad
  \mu^{(k-1)}=d\mu^{(k-2)}(\eta).
\]
Then these functions $\mu',\mu''\cdots,\mu^{(k-1)}$
are $C^\infty$-functions on $S_2$.

\begin{corollary}\label{cor:submanifold2}
 Let $U$ be a domain in $\K^n$, 
 $f\colon{}U\to\K^{n+1}$ a front, 
 and $p\in U$ a $2$-nondegenerate singular point.
 Then $p$ is $\K$-right-left equivalent to
 the $A_{k+1}$-singularity $(3\le k\le n)$
 if and only if
 \begin{equation}\label{eq:mu-cond}
   \mu=\mu'=\dots = \mu^{(k-2)}=0,\qquad
   \mu^{(k-1)}\neq 0\qquad \text{at $p$}
 \end{equation}
 hold, and the Jacobian matrix of the map 
 $\Phi_\mu:=(\mu',\dots,\mu^{(k-2)})\colon{}S_2\to\K^{k-2}$ is of rank
 $k-2$ at $p$.
\end{corollary}

\begin{remark}
 In Corollary \ref{cor:submanifold2},
 the criterion for $A_4$-singularities reduces to
 $2$-nondegeneracy and
 \[
   \mu(p)=\mu'(p)=0,\qquad
   \mu''(p)\neq 0,
 \]
 since the condition $\mu''(p)\neq 0$
 implies that 
 the map $\Phi_\mu=\mu'$ is of rank $1$ at $p$.
\end{remark}
Proofs of these assertions are given in Section~\ref{sec:versal}.

Let $f:U\to\K^{n+1}$ be a front and
$p$ a $1$-nondegenerate singular point.
Then we can consider the restriction 
$f_1:S(f)\to\K^{n+1}$ of $f$ into $S_1$.
Let us denote the limiting tangent hyperplane by 
\[
   \nu(p)^\perp:=\{\vect v\in T_{f(p)}\K^{n+1}\,;\, 
                   \inner{\vect v}{\nu(p)}=0\}.
\]
The following assertion can be proved straightforwardly.
\begin{corollary}\label{thm:shadow}
 Let $f:U\to \K^{n+1}$  be a front and $p\in U$ 
 a $1$-nondegenerate singular point.
 For any vector $\vect n \not\in \nu(p)^\perp$,
 the following are equivalent{\rm:}
 \begin{enumerate}
  \item  $f$ at $p$ is an $A_{k+1}$-front singularity $(k\leq n)$.
  \item  The projection $\pi_{\vect n}\circ f|_{S(f)}$ at $p$ is an
	$A_{k}$-front singularity $(k\leq n)$.
 \end{enumerate}
 Here, 
 $\pi_{\vect n}$ is the normal projection with respect to $\vect n$
 to the hyperplane 
 \[
    \vect n^\perp:=\{\vect v\in T_{f(p)}\K^{n+1}\,;\, 
        \inner{\vect v}{\vect n}=0\}
 \]
 and an $A_1$-front singularity means a regular point.
\end{corollary}

The following assertion follows immediately from 
Theorem~\ref{thm:morin-cri} in the appendix.
\begin{corollary}\label{cor:rest-front}
 Let $\Omega$ be a domain in $\K^{n+1}$ and $f:\Omega\to \K^{n+1}$
 a $\K$-differentiable map.
 Suppose that $p\in \Omega$ is a $1$-nondegenerate singular point,
 namely, the exterior derivative of the Jacobian
 of $f$ does not vanish at $p$.
 Then the following are equivalent{\rm:}
 \begin{enumerate}
  \item $p$ is an $A_k$-Morin singular point of $f$.
  \item $f|_{S(f)}$ is a front, and 
    $p$ is an $A_k$-front singularity of $f|_{S(f)}$.
 \end{enumerate}
\end{corollary}
Here, the $A_k$-Morin singularities are defined in the appendix, and 
the $A_0$-Morin singularity means a regular point.
\section{Adopted coordinates and $k$-nondegeneracy}\label{sec:adopted}

In \cite{krsuy}, a certain kind of special coordinate system was
introduced to treat the recognition problem of 
cuspidal edges and swallowtails in $\R^3$.
In this section, we give a generalization of such a coordinate system
around a singular point, which will play a crucial role.
\begin{lemma}[Existence of the $k$-th KRSUY-coordinate system]%
 \label{lem:adopted-coord} 
 Let $f:U\to\K^{n+1}$ be a front and
 $p\in U$  a $k$-nondegenerate singular point $(1\le k \le n)$.
 Then there exists a $\K$-differentiable coordinate system
 $(z_1,\dots,z_n)$ around $p$
 and a $($non-degenerate$)$ 
 $\K$-affine transformation $\Theta:\K^{n+1}\to\K^{n+1}$
 such that
 \[
   \hat{f}(z_1,\ldots,z_n)
        =\left(\hat{f}^{1},\ldots,\hat{f}^{n+1}\right)
        =\Theta\circ f(z_1,\ldots,z_n)
 \]
 satisfies the following properties{\rm :}
 \begingroup
 \renewcommand{\theenumi}{{\rm($\mathrm{P\!}_{\arabic{enumi}}$)}}
 \renewcommand{\labelenumi}{{\theenumi\ }}
 \begin{enumerate}
  \setcounter{enumi}{-1}
  \item\label{p:krsuy-point}
       The point $p$ corresponds to $(z_1,\dots,z_n)=(0,\dots,0)$.
  \item\label{p:krsuy-null} 
       $\partial_{z_n}:=\partial/\partial z_n$ gives 
       an extended null vector field on $U$.
  \item\label{p:krsuy-span1}
       If $k\ge 2$, 
       then the $\K$-vector space $T_pS_i$ $(i=1,\ldots,k-1)$ is spanned by
       $\partial_{z_{i+1}},\dots,\partial_{z_{n}}$, 
       that is
       \[
          T_pS_i=\spann{\partial_{z_{i+1}},\dots,\partial_{z_{n}}}, 
       \]
where
       $\partial_{z_{j}}:=\partial/\partial {z_{j}}$, $j=1,\dots,n$.
  \item\label{p:krsuy-span2}
       Suppose $k<n$. 
       If $p$ is $(k+1)$-singular, 
       then  $T_pS_k=\spann{\partial_{z_{k+1}}\dots,\partial_{z_n}}$.
       On the other hand, 
       if $p$ is not $(k+1)$-singular, 
       $T_pS_k=\spann{\partial_{z_k},\dots,\partial_{z_{n-1}}}$.
  \item\label{p:krsuy-init}
       $\hat f(p)=\vect 0$ and  $\nu(p)=(1,0,\ldots,0)$.
  \item\label{p:krsuy-coord}
       $\hat{f}(z_1,\dots,z_n)=
         \left(
           \hat f^{1}(z_1,\dots,z_n),
           \hat f^{2}(z_1,\dots,z_n),z_1,\ldots,z_{n-1}
         \right)$
       for $n\ge 2$
       and 
       $\hat{f}(z_1)=
         \left(
           \hat f^{1}(z_1),
           \hat f^{2}(z_1)\right)$ for $n=1$.
  \item\label{p:krsuy-deriv}
       If $n\ge 2$, 
       $(\hat f^{1})_{z_i}(p)=(\hat f^{2})_{z_i}(p)=0$
       holds for  $i=1,\dots,n-1$.
 \end{enumerate}
\endgroup
\end{lemma}
We call a coordinate system on $U$ satisfying
the properties \ref{p:krsuy-point}, \ref{p:krsuy-null}, \ref{p:krsuy-span1}
and \ref{p:krsuy-span2}
a ``$k$-{\em adopted coordinate system}'',
and we call by  
the ``{\em $k$-th KRSUY-coordinate system}''
a pair of coordinate systems on $U$ and $\K^{n+1}$ satisfying
all above conditions.
(In \cite{krsuy}, the existence of the coordinates for $k=1$, $2$ 
and $n=2$ was shown.)

\begin{proof}[Proof of Lemma~\ref{lem:adopted-coord}]
 By replacing $U$ to be a smaller neighborhood if necessary,
 we can take an extended null vector field $\tilde \eta$ on 
 $U$ (see Remark \ref{rem:null}).
 Now we assume that $p\in S_k$ is $k$-nondegenerate.
 Then $S_k$ is a submanifold of $U$ of codimension $k$.
 So we can take a basis $\{\xi_{1},\ldots,\xi_{n-1},\tilde \eta_p\}$
 of $T_{p}U$ 
 such that
 \[
     \begin{cases}
        \spann{\xi_{k},\ldots,\xi_{n-1}}=T_pS_k &
        \text{if $k<n$ and $p$ is not $k$-singular}, \\
        \spann{\xi_{k+1},\ldots,\xi_{n-1},\tilde \eta_{p}}=T_pS_k &
        \text{if $k<n$ and $p$ is $k$-singular}, 
     \end{cases}
 \]
 and
 \[
     \spann{\xi_{i+1},\ldots,\xi_{n-1},\tilde \eta_{p}}=T_pS_i
         \qquad (1\le i \le k-1\quad \text{and}\quad k\ge 2).
 \]
 We now take a local coordinate system $(s_1,s_2,\dots, s_n)$ around $p$
 such that $\partial_{s_i}(p)=\xi_i$ ($i=1,\dots,n-1$),
 and consider the following two involutive distributions 
 $T_1$ and $T_2$ of rank $n-1$ and $1$ respectively
 \[
     T_1(x):=
      \spann{\partial_{s_1}(x),\partial_{s_2}(x),\dots,\partial_{s_{n-1}}(x)},
     \quad
     T_2(x):=\spann{\tilde \eta_x}
     \quad (x\in V),
 \]
 where $V$ is a sufficiently small neighborhood of $p$.
 Since two distributions $T_1(p)$ and $T_2(p)$ span $T_pV$,
 the lemma in \cite[page 182]{KN} yields
 the existence of local coordinate system 
 $(w_1,w_2,\dots,w_n)$ around $p$ such that
 \[
     T_1:=\spann{\partial_{w_1},\partial_{w_2},\dots,\partial_{w_{n-1}}},
     \qquad
     \partial_{w_n}\in \spann{\tilde \eta}(=T_2).
 \]
 Moreover, by a suitable affine transformation
 of $\K^n$, we may assume that
 $w_i(p)=0$ ($i=1,\dots,n$) and 
 $\partial_{w_i}(p)=\xi_i$ ($i=1,\dots,n-1$).
 Replacing $U$ to be sufficiently smaller,
 we may assume that this new coordinate system  $(w_1,\dots,w_n)$ 
 is defined on $U$.
 Moreover, we may reset
 $\tilde \eta=\partial_{w_n}$.
 Thus $(U;w_1,\dots,w_n)$ satisfies
 \ref{p:krsuy-point}, \ref{p:krsuy-null}, \ref{p:krsuy-span1} 
 and \ref{p:krsuy-span2}.
 By a suitable affine transformation of  $\K^{n+1}$,
 we may assume that
 \begin{equation}\label{eq:f-position}
     f(p)=\vect{0},\quad
     \nu(p)=\vect{e}_1,\qquad
     \frac{\partial f}{\partial w_i}(p)=\vect{e}_{i+2}
        \quad (i=1,\dots,n-1),
 \end{equation}
 where $\vect{e}_j$ denotes the $j$-th canonical vector
\begin{equation}\label{eq:e}
      \vect{e}_j:=(0,\dots,0,1,0,\dots,0)\qquad (j=1,2,\dots,n+1)
\end{equation}
 of $\K^{n+1}$. 
 We set $f=(f^1,\dots,f^{n+1})$.
 Then it holds that
 \begin{equation}\label{eq:f-init-deriv}
      \frac{\partial f^{j+2}}{\partial w_l}(p) = \delta^j_l
        \qquad (j,l=1,\dots,n-1),
 \end{equation}
 where $\delta^j_l$ is Kronecker's delta.
 Define new functions by
 \[
    g_j(z_1,\dots,z_{n-1};w_1,\dots,w_n):= f^{j+2}(w_1,\dots,w_n) - z_j
    \qquad (j=1,\dots,n-1).
 \]
 Since $\partial g_j/\partial w_l(0)=\delta^j_l$ for
 ($j,l=1,\dots,n-1$),
 the implicit function theorem yields that 
 there exist functions 
 $\varphi_j(z_1,\dots,z_{n-1},w_n)$ ($j=1,\dots,n-1$) such that
 \[
     g_j\bigl(
        z_1,\dots,z_{n-1},
         \varphi_1(z_1,\dots,z_{n-1},w_n),\dots,
         \varphi_{n-1}(z_1,\dots,z_{n-1},w_n),w_n\bigr)=0.
 \]
 If we set $w_n:=z_n$ and
 \[
     w_j :=\phi_j(z_1,\dots,z_{n-1},w_n) \qquad (j=1,\dots,n-1),
 \]
 then $(z_1,\dots,z_n)$
 gives a new local coordinate system of $U$ around $p$ such that
 \begin{align}
 \label{eq:implicit}
    z_j &= f^{j+2}\bigl(\varphi_1(z_1,\dots,z_{n-1},z_n),
           \dots,\varphi_{n-1}(z_1,\dots,z_{n-1},z_n),z_n\bigr),\\
 \label{eq:jacobi-z}
   \frac{\partial w_j}{\partial z_l}(p) &= \delta^j_l 
    \qquad (j=1,\dots,n-1,~ l=1,\dots,n).
 \end{align}
 Since $\partial_{w_n}$ is the null direction,
 we have $\partial f/\partial w_n=0$.
 Differentiating \eqref{eq:implicit} with respect to $z_n$,
 we have
 \[
     0 = \sum_{l=1}^{n-1}\frac{\partial f^{j+2}}{\partial w_l}
          \frac{\partial \varphi_l}{\partial z_n}
          + \frac{\partial f^{j+2}}{\partial w_n}
       = \sum_{l=1}^{n-1}\frac{\partial f^{j+2}}{\partial w_l}
         \frac{\partial \phi_{l}}{\partial z_n}
     \qquad\text{on $S_1$}.
 \]
 Since the matrix 
 $\bigl(\partial f^{j+2}/\partial w_l\bigr)_{j,l=1,\dots,n-1}$
 is regular near $p$ by \eqref{eq:f-init-deriv},
 \[
     \frac{\partial w_l}{\partial z_n}=
     \frac{\partial\varphi_l}{\partial z_n}=0\qquad
                          \text{(on $S_1$ near $p$)}
 \]
 hold for $l=1,\dots,n-1$.
 In particular, \eqref{eq:jacobi-z} holds for $j=n$.

 Thus if we set 
 \[
     \hat f(z_1,\dots,z_n) :=
       f\bigl(
          \varphi_1(z_1,\dots,z_n),\dots
          \varphi_{n-1}(z_1,\dots,z_n),
           z_n
         \bigr),
 \]
 then 
 \[
      \frac{\partial \hat f}{\partial z_n} 
                = \sum_{j=1}^{n-1}
                      \frac{\partial f}{\partial w_j}
                  \frac{\partial \varphi_j}{\partial z_n}
                   + \frac{\partial f}{\partial w_n} = \vect{0}
           \qquad \text{on $S_1$ near $p$},
 \]
 holds which implies \ref{p:krsuy-null}.
 By \eqref{eq:jacobi-z}, $\partial_{z_j}(p)=\partial_{w_j}(p)$ for
 $j=1,\dots,n$.
 Hence we have \ref{p:krsuy-span1} and \ref{p:krsuy-span2}.
 Moreover,
 \[
    \frac{\partial \hat f}{\partial z_i}(p) = 
       \left(
         \sum_{j=1}^{n-1}\frac{\partial f}{\partial w_j}(p)
                         \frac{\partial w_j}{\partial z_i}(p)
       \right)
      +\frac{\partial f}{\partial w_{n}}(p) 
      =\frac{\partial f}{\partial
                        w_{i}}(p)=\vect{e}_{i+2},
 \]
 which implies \ref{p:krsuy-deriv}.
 Now \ref{p:krsuy-init} and \ref{p:krsuy-coord}
 follow immediately from  \eqref{eq:f-position} and \eqref{eq:implicit} 
 respectively.
\end{proof}

Now we use the notation
${~}'=\partial/\partial {z_n}$ under the $k$-adopted coordinate system.
\begin{lemma}
  \label{lem:adopted-kie}
  Let $p$ be a $k$-nondegenerate singular point of $f$.
  Under a $k$-th adopted coordinate system $(z_{1},\dots,z_{n})$,
  the following hold{\rm :}
 \begin{enumerate}
  \item\label{item:kie1}
        If $k=1$, then $p$ is\ $2$-singular if and only if
    $\lambda'(p)=0$.
    Moreover $p(\in S_2)$ is $2$-nondegenerate 
    if and only if 
    \[
      \bigl(
       \lambda'_{z_2}(p),\ldots,\lambda'_{z_{n-1}}(p),\lambda''(p)
      \bigr)
      \neq
      \vect{0}.
    \]
  \item\label{item:kie2}
        If $k\geq 2$, then $\lambda_{z_{m-1}}^{(m-2)}(p)\ne0$
    and
    \begin{equation}\label{eq:m}
     \lambda^{(m-2)}_{z_m}(p)
      =\dots=
      \lambda_{z_{n-1}}^{(m-2)}(p)=\lambda^{(m-1)}(p)=0
    \end{equation}
    hold for $m=2,\dots,k$.
    Moreover
    $p$ is $(k+1)$-singular $(k<n)$
    if and only if\/ $\lambda^{(k)}(p)=0$.
    Furthermore $p(\in S_{k+1})$ is $(k+1)$-nondegenerate if and
    only if
    \[
       \bigl(
        \lambda^{(k)}_{z_{k+1}}(p),\dots,
            \lambda^{(k)}_{z_{n-1}}(p),\lambda^{(k+1)}(p)
       \bigr)\neq
       \vect{0}.
    \]
 \end{enumerate}
\end{lemma}

\begin{proof}
 When $1 \le k <n$, $p$ is $(k+1)$-singular if and only if
 $\lambda^{(k)}(p)=0$ by \eqref{eq:lambda-eta}.
 Suppose now $p$ is $(k+1)$-singular ($k<n$).
 By definition, $p$ is $(k+1)$-nondegenerate if and only if
 $T_pS_{k}\not \subset \ker(d\lambda^{(k)})_p$.
 Since 
 $\spann{\partial_{z_{k+1}},\dots,\partial_{z_{n}}} =T_pS_{k}$ 
 by \ref{p:krsuy-span1},
 this is equivalent to
 $\bigl(
   \lambda^{(k)}_{z_{k+1}},\dots,\lambda^{(k)}_{z_{n-1}},
    \lambda^{(k+1)}\bigr)(p)\neq\vect{0}$,
 which proves \ref{item:kie1} and the latter part of \ref{item:kie2}.
 Next, we assume $2\le m\le k\le n$
 and prove the first part of \ref{item:kie2}:
 By \eqref{eq:lambda-eta},
 we have
 $S_{m-1}=\{q\in S_{m-2}\,;\,\lambda^{(m-2)}(q)=0\}$.
 By the assumption \ref{p:krsuy-span1}, we have
 $T_pS_{m-1}=\spann{\partial_{z_m},\dots,\partial_{z_n}}$.
 In particular, $\lambda^{(m-2)}$ is constant along these
 directions, which implies \eqref{eq:m}. 
 On the other hand, since $k\ge m\ge 2$,
 $p$ is $(m-1)$-nondegenerate.
 By \eqref{eq:not}, we have
 $T_pS_{m-2}\not \subset \ker(d\lambda^{(m-2)})_p$,
 which implies $\lambda_{z_{m-1}}^{(m-2)}(p)\neq 0$
 because of \eqref{eq:m}.
\end{proof}

\begin{lemma}\label{lem:adopted-f}
 Suppose that $p\in U$ is a singular point of $f$
 which is  $k$-nondegenerate but not $(k+1)$-singular.
 Then under the $k$-th adopted coordinate system,
 $f^{(k+1)}(p)\ne\vect{0}$ holds. 
 Moreover, if $k\ge 2$, then $f^{(m-1)}_{z_{m-1}}(p)\ne\vect{0}$ and
 \begin{equation}\label{eq:zf}
    f^{(m-1)}_{z_m}(p)=\dots=f^{(m-1)}_{z_{n-1}}(p)=f^{(m)}(p)=\vect{0}
 \end{equation}
 hold for $m=2,\dots,k$.
\end{lemma}
\begin{proof}
 Since $p\in S_k\setminus S_{k+1}$, \eqref{eq:lambda-eta}
 yields that
 \[
    f'(p)=\dots=f^{(k)}(p)=\vect{0},\quad f^{(k+1)}(p)\ne \vect{0}.
 \]
 Suppose that $k\geq m\geq2$. 
 Then $p$ is $m$-nondegenerate.
 By \eqref{eq:lambda-eta}, 
 the map $f^{(m-1)}$ vanishes along $S_{m-1}$.
 Since 
 $T_pS_{m-1}$ is spanned by $\partial_{z_m},\dots,\partial_{z_n}$
 because of \ref{p:krsuy-span1},
  we have \eqref{eq:zf}.
 On the other hand, since $f'=\dots =f^{(m-1)}=\vect{0}$ at $p$,
 Lemma~\ref{lem:adopted-kie} yields that
 \[
    0\ne \lambda_{z_{m-1}}^{(m-2)}(p)
     =\det(f_{z_1},\cdots,f_{z_{n-1}},f_{z_{m-1}}^{(m-1)},\nu)(p),
 \]
 which implies that $f_{z_{m-1}}^{(m-1)}(p)\ne \vect 0$.
\end{proof}
\section{Proof of the criteria}\label{sec:versal}
Let $f\colon{}U\to\K^{n+1}$ be a front and $\nu$ be 
a $\K$-normal vector.
Let $p\in U$ be a singular point of $f$ which is
$k$-nondegenerate but not $(k+1)$-singular.
We denote the $k$-th KRSUY coordinates
by $(U;z_1,\ldots,z_n)$ and
$(\K^{n+1};x_1,\dots,x_{n+1})$. Set
\[
    X:=(x_1,x_2),\quad Y:=(z_1,\dots,z_{n-1}),\quad
    (X,Y)=(x_1,x_2,z_1,\dots,z_{n-1}).
\]
We define a function $\Phi\colon{}\K^{n+1}\times\K\to\K$ by
\begin{equation}\label{eq:Phi-def}
    \Phi(X,Y,z_n) :=
          \inner{\nu(Y,z_n)}{f(Y,z_n)-(X,Y)}.
\end{equation}
Then we have the following
\begin{proposition}\label{lem:discri}
 The discriminant set 
 \[
    D_{\Phi}:=\left\{
                (X,Y)\in \K^{n+1}\,;\, 
                \Phi(X,Y,t)=\frac{\partial \Phi}{\partial t}(X,Y,t)=0
                ~\text{for some $t$}
              \right\}
 \]
 coincides with the image $\Image(f)$ of $f$.
\end{proposition}

\begin{proof}
 One can easily prove that $D_{\Phi}\supset\Image(f)$.
 Next we show $D_{\Phi}\subset\Image(f)$.
 Suppose  $(X,Y)\in D_{\Phi}$.
 Since $\inner{f_{z_n}}{\nu}=0$,
 $(X,Y)\in D_{\Phi}$ is equivalent to
 \begin{equation}\label{eq:lem-dis1}
    \inner{\nu}{f-(X,Y)}=0 \quad \text{and}\quad
    \inner{\nu'}{f-(X,Y)}=0.
 \end{equation}
 Since we are using the $k$-th KRSUY-coordinates,
 \eqref{eq:lem-dis1} reduces to
 \begin{equation}\label{eq:lem-dis2}
    \inner{\nu}{(f^1-x_1,f^2-x_2,0,\dots,0)}=
     \inner{\nu'}{(f^1-x_1,f^2-x_2,0,\dots,0)}=0.
 \end{equation}
 By the following Lemma~\ref{lem:lin-indep},
 the first two components of $\nu,\nu'$ are linearly independent near 
 $p$ as vectors in $(\K^2;x_1,x_2)$.
 Hence by \eqref{eq:lem-dis2},
 we have  $x_1=f^1$, $x_2=f^2$.
 Since
 \[
    (X,Y)=(f^1(Y,z_n),f^2(Y,z_n),Y)=f(Y,z_n) \in\Image(f)
 \]
 holds, we have $D_{\Phi}\subset\Image(f)$.
\end{proof}
\begin{lemma}\label{lem:lin-indep}
 The  vectors $f_{z_1},\ldots,f_{z_{n-1}},\nu,\nu'$ are
 linearly independent near $p$.
 Moreover, if we write $\nu=(\nu^1,\dots,\nu^{n+1})$, then
 $(\nu^1,\nu^2)$ and $\bigl((\nu^1)',(\nu^2)'\bigr)$
 are linearly independent in $\K^2$ near $p$.
\end{lemma}
\begin{proof}
 Clearly, $f_{z_1},\ldots,f_{z_{n-1}},\nu$
 are linearly independent.
 Since $f$ is a front and $f'(p)=\vect{0}$, 
 we have $\nu'\not\in\spann{\nu}$.
 We assume $\nu'\in\spann{f_{z_1},\ldots,f_{z_{n-1}},\nu}$ 
 at $p$.
 Since 
 $\inner{\nu'}{f_{z_i}} =-\inner{\nu}{f_{z_i}'}
  =\inner{\nu_{z_i}}{f'}=0$,
 $\nu'\in\spann{\nu}$, a contradiction. 
 Moreover, since $f_{z_j}(p)=\vect{e}_{j+2}$ by \ref{p:krsuy-coord},
 $(\nu^1,\nu^2)$ and 
 $\bigl((\nu^1)',(\nu^2)'\bigr)$ are linearly independent.
 Then the second part of the lemma follows.
\end{proof}
\begin{lemma}\label{lem:phi2lam}
 Suppose that $p\in S(f)$  is $k$-nondegenerate
 but not $(k+1)$-singular.
 Under the  $k$-th KRSUY-coordinates at $p$,
 it holds that
 \begin{enumerate}
  \item\label{item:phi2lam:1}
       $\Phi^{(m)}(\vect 0, 0)=0$ $(m=0,\dots, k+1)$
       and  $\Phi^{(k+2)}(\vect{0}, 0)\ne0$, where
       $\vect 0$ is the origin of $\K^{n+1}$,
  \item\label{item:phi2lam:2}
       the Jacobian matrix of the map
       $(\Phi,\Phi',\dots,\Phi^{(k)})
         :\K^{n+2}\to \K^{k+1}$ is of rank $k+1$ at the origin.
 \end{enumerate}
\end{lemma}
\begin{proof}
 By identifying $\bigwedge^n \K^{n+1}$ with $\K^{n+1}$
 with respect to $\inner{~}{~}$,
 $f_{z_1},\dots,f_{z_{n-1}},\nu$ and $f_{z_1}\wedge\dots\wedge
 f_{z_{n-1}}\wedge \nu$  are linearly independent near $p$, 
 where $\wedge$ denotes
 the exterior product.
 Then, one can write
 \begin{equation}\label{eq:nu-prime}
  \nu' = -a f_{z_{1}}\wedge  \dots \wedge
                  f_{z_{n-1}}\wedge \nu + 
          \left(\sum_{l=1}^{n-1} b_l f_{z_l}\right)
          +c\nu,
 \end{equation}
 for some functions $a$, $b_1$, \dots, $b_{n-1}$, $c$ 
 in $(z_1,\dots, z_n)$.
 By Lemma~\ref{lem:lin-indep}, we have
 \begin{equation}\label{eq:alpha}
     a  \neq 0 \qquad \text{near $p$}.
 \end{equation}
 Here, on the singular set $S_1$, we have
 \[
    0 = \inner{\nu_{z_j}}{f'}
      = \inner{\nu'}{f_{z_j}}
      =\sum_{l=1}^{n-1}b_l\inner{f_{z_l}}{f_{z_j}}
      \qquad (j=1,\dots,n-1).
 \]
 Hence we have
 \begin{equation}\label{eq:alpha-beta}
     b_j =0  \quad(j=1,\dots,n-1)
             \qquad \text{on $S_1$ near $p$}.
 \end{equation}
 By the definition of $\lambda$, we can write
 \begin{equation}\label{eq:nu1-f1}
    \inner{\nu'}{f'} = a\lambda + 
               \sum_{l=1}^{n-1}b_l\inner{f_{z_l}}{f'}.
 \end{equation}
 Since $\inner{\nu'}{f'}=0$,
 differentiating \eqref{eq:Phi-def} with respect to $z_n$, we have
 \begin{align}
   \Phi'&=
       \inner{\nu'}{f-(X,Y)}+\inner{\nu}{f'} =
       \inner{\nu'}{f-(X,Y)},\nonumber\\
   \Phi''&=
       \inner{\nu''}{f-(X,Y)}+2\inner{\nu'}{f'} +\inner{\nu}{f''}
        \label{eq:phi-2prime}\\
       &= 
       \inner{\nu''}{f-(X,Y)}+\inner{\nu'}{f'}.
   \nonumber
 \end{align}
 Then by Lemmas~\ref{lem:adopted-kie}, \ref{lem:adopted-f} 
 and  \eqref{eq:nu1-f1},  
 we have
 \begin{equation}\label{eq:diamond}
  \Phi^{(m+2)} = \left(\inner{\nu'}{f'}\right)^{(m)}
           = \left(a\lambda + 
                          \sum_{l=1}^{n-1}b_l
                           \inner{f_{z_l}}{f'}
                          \right)^{(m)}= a\lambda^{(m)}
 \end{equation}
 at $(\vect{0},0)$ for $m=0,\dots,k$, 
 where we used \eqref{eq:alpha-beta} for $m=k$.  
 Since $\lambda^{(m)}(p)=0$ ($m\le k-1$) and
 $\,\lambda^{(k)}(p)\ne 0$,
 \eqref{eq:alpha} yields the first part of the lemma.

 Next, we show the second part:
 Differentiating \eqref{eq:phi-2prime} with respect to $z_j$
 \[
   \Phi''_{z_j} =
    \inner{\nu_{z_j}}{f-(X,Y)}+\inner{\nu''}{f_{z_j}-\vect{e}_{j+2}}+
    \frac{\partial}{\partial z_{j}}\inner{\nu'}{f'},
 \]
 where $\vect{e}_{j+2}$ is defined in \eqref{eq:e}.
 Like as in \eqref{eq:diamond}, 
 applying Lemma~\ref{lem:adopted-kie}, 
 we have
 \[
     \Phi^{(m)}_{z_j}(\vect{0},0) 
           = a \lambda^{(m-2)}_{z_j}
            =\begin{cases}
            0 \qquad &(j>m-1),\\
            a(p) \lambda^{(m-2)}_{z_{m-1}}(p)\neq 0\quad &(j=m-1).
         \end{cases}
 \]
  Since 
 $\Phi_{x_j} = \inner{\nu}{-\vect{e}_j}=-\nu^j$ for $j=1,2$, 
 the Jacobian matrix of $(\Phi,\Phi',\dots,\Phi^{(k)})$ 
 at the origin is written as 
 \[
    \begin{pmatrix}
      -\nu^1 & -\nu^2 & 0 & 0 & \dots & 0 & \dots & 0 \\
      -(\nu^1)' & -(\nu^2)' & 0 & 0 & \dots & 0 &\dots & 0 \\
     *      &  * & a\lambda_{z_1} & 0 & \dots & 0 & \dots &0 \\
     *      &  * & * & a\lambda'_{z_2}& \dots & 0 & \dots &0 \\
     \vdots & \vdots & \vdots & \vdots & \ddots & \vdots & \vdots & \vdots\\
     *      &  * & * &   * &      * & 
     a\lambda^{k}_{z_{k-1}} &\dots & 0 
    \end{pmatrix},
 \]
 which is of rank $k+1$.
\end{proof}
\begin{proof}[Proof of Theorem~\ref{thm:criteria}]
 It can be straightforwardly checked  that the singular point
 given in \eqref{eq:ak-def} is $k$-nondegenerate, but not
 $(k+1)$-singular.
 We now prove the converse.
 By Zakalyukin's theorem \cite{zaka} (see also the appendix of
 \cite{krsuy}), 
 $\Image(f)$ is locally diffeomorphic to the standard
 $A_{k+1}$-singularity if and only if it is $\K$-right-left equivalent
 to \eqref{eq:ak-def}, whenever
 the regular points of $f$ are dense.
 Thus, it is sufficient to show the versality of $\Phi$
 (which implies that $D_\Phi$ is locally diffeomorphic
 to the standard $A_{k+1}$-singularity, see \cite{bruce-giblin}).
 In fact, this is evident by Lemma~\ref{lem:phi2lam}.
\end{proof}

\begin{remark*}
 In \cite{krsuy}, versal unfoldings of 
 a  cuspidal edge and a swallowtail in $\R^3$ are given. 
 The above versal unfolding $\Phi$ is 
 much simpler than those in \cite{krsuy}.
\end{remark*}

Now, we give  proofs of Corollaries \ref{cor:actual},
\ref{cor:submanifold1} and \ref{cor:submanifold2}
using Theorem~\ref{thm:criteria}.
\begin{proof}[Proof of Corollary \ref{cor:actual}]
 The necessary part is obvious by Lemma \ref{lem:adopted-kie}.
 We assume that $\lambda=\lambda'=\cdots=\lambda^{(k)}=0$,
 $\lambda^{(k+1)}\ne0$ at $p$ and
 the map $(\lambda,\lambda',\ldots,\lambda^{(k)})$
 is non-singular at $p$.
 Then we can see that $p$ is 1-nondegenerate, so we can
 take the $1$-st adopted coordinate system.
 By Lemma \ref{lem:adopted-kie},
 $p$ is $2$-singular if $k\geq2$.
 Since $(\lambda,\lambda')$ is full rank,
 $p$ is $2$-nondegenerate.
 We can continue this argument $k$ times.
 Then we see that $p$ is $k$-nondegenerate
 but not $(k+1)$-singular.
 By Theorem~\ref{thm:criteria}, we have the conclusion.
\end{proof}
\begin{proof}[Proof of Corollary \ref{cor:submanifold1}]
 By Lemma~\ref{lem:zero}, $p$ is $2$-singular if and only if 
 $\eta\in T_pS(f)=\spann{v_1,\dots,v_{n-1}}$, which 
 is equivalent to $\mu(p)=0$.
 Hence $p$ is $A_2$-singular point if and only if $\mu(p)\neq 0$.

 Since the condition does not depend on the choice of coordinate
 system and a frame $\{v_1,\dots,v_{n-1}\}$ on $S(f)$,
 we take a coordinate system $(x_1,\dots,x_n)$ of $U$ 
 such that $x_n=\lambda$ and set $v_j=\partial/\partial x_j$
 ($j=1,\dots,n-1$).
 Then the singular set $S(f)$ is given as a hyperplane $x_n=0$, and
 $\mu=\eta_n$ holds on $S(f)$, where $\eta=(\eta_1,\dots,\eta_n)$.
 On the other hand, $\lambda'=d\lambda(\eta)=dx_n(\eta)=\eta_n$,
 and then we have
 $\lambda'=\mu$ on $S(f)$.
 Then, if $p$ is $2$-singular, it is $2$-nondegenerate if and only if
 $d\lambda'(T_p S(f))=d\mu(T_pS(f))\neq 0$.
 Moreover, since 
 \[
    S_2=\{q\in S(f)\,;\, \lambda'(q)=0\}=\{q\in S(f)\,;\,\mu(q)=0\},
 \]
 $p$ is $3$-singular if and only if $d\mu(\eta_p)=0$.
 Thus, we have a criterion for $A_3$-singularities.
\end{proof}

\begin{proof}[Proof of Corollary~\ref{cor:submanifold2}]
 Since $p$ is $2$-nondegenerate, $d\mu_p\neq 0$ holds by 
 Corollary~\ref{cor:submanifold1}.
 Hence, in addition to the previous proof, one can take a coordinate
 system $(x_1,\dots,x_n)$ on $U$ such that $\lambda=x_n$ and
 $\mu=x_{n-1}$ on $S(f)$.
 Then the set $S_2$ is a linear subspace  $\{(x_1,\dots,x_{n-2},0,0)\}$
 near $p$.
 As seen in the previous proof, $\lambda'=\mu$ holds on $S(f)$,
 and hence the condition \eqref{eq:actual-1} in
 Corollary~\ref{cor:actual} is equivalent to \eqref{eq:mu-cond}
 under the assumption that $p$ is $2$-singular.
 Moreover, since $\lambda''=\mu'$ on $S_2$, 
 the Jacobian matrix of the map $\Lambda$ in Corollary~\ref{cor:actual}
 is computed as
 \[
 \begin{small}
     \begin{pmatrix}
      \lambda_{x_1} &  \dots & \lambda_{x_{n-2}} &
      \lambda_{x_{n-1}} &
      \lambda_{x_{n}}\\
      \lambda'_{x_1}& \dots & \lambda'_{x_{n-2}} 
      & \lambda'_{x_{n-1}} &
      \lambda'_{x_{n}}\\
      \lambda''_{x_1}& \dots & \lambda''_{x_{n-2}} 
      & \lambda''_{x_{n-1}} &
      \lambda''_{x_{n}}\\
        \vdots & \ddots &\vdots & \vdots & \vdots\\
      \lambda^{(k-1)}_{x_1} & \dots 
      & \lambda^{(k-1)}_{x_{n-2}} & \lambda^{(k-1)}_{x_{n-2}}
      & \lambda^{(k-1)}_{x_{n}}
     \end{pmatrix}
     = 
     \begin{pmatrix}
       0 &  \dots&0 & 0 & 1 \\
       0 &  \dots&0 & 1 & * \\
       \mu'_{x_1} & \dots & \mu'_{x_{n-2}}
       & * & * \\
        \vdots &  \ddots &\vdots & \vdots & \vdots\\
      \mu^{(k-1)}_{x_1} & \dots & \mu^{(k-1)}_{x_{n-2}}
       & * & * 
     \end{pmatrix}
  \end{small}
 \]
 holds at $p$.  Hence the conclusion follows.
\end{proof}

\begin{example}[Tangent developables in $\K^4$]
 Let $\gamma(z)$ be a $\K$-differentiable curve
 in $\K^4$ such that $\gamma',\gamma'',\gamma''',\gamma''''$
 are linearly independent.
 We consider a hypersurface
 $f(z,u,v)=\gamma(z)+u\gamma'(z)+v\gamma''(z)$.
 If we set
 $\nu=\gamma'\wedge\gamma''\wedge\gamma'''$,
 then  $f$ is a front with $\K$-normal vector $\nu$.
 Furthermore, we have
 \[
     S_1=\{v=0\},\quad
     S_2=\{u=v=0\},\quad
     \eta=(-1,1,u)\quad\text{and}\quad
     \lambda=v.
 \]
 By Corollary~\ref{cor:actual},
 $f$ at $(z,u,0)$ is $\K$-right-left equivalent to $\text{(cusp)}\times\K^2$
 if and only if $u\ne0$, and
 $f$ at $(z,0,0)$ is  $\K$-right-left equivalent to
 $\text{(swallowtail)}\times\K$. 
\end{example}

\section{Zigzag numbers on A-fronts}\label{sec:zigzag}

Let $M^n$ (resp.\  $N^{n+1}$) be a 
manifold of dimension $n$ (resp.\ $n+1$),
and $f\colon{}M^n\to N^{n+1}$ a front.
If all singularities of $f$ are $A_k$-singular points $(k\leq n+1)$, 
it is called an {\em $A$-front}.
In this section, we set $\K=\R$, and discuss
a topological invariant of $A$-fronts.
Now, we assume that $M^n$ is orientable.

For each $A_k$ $(k\leq n+1)$ singular point on an $A$-front,
the image of it coincides with the discriminant set $\{F=F_t=0\}$ of 
\eqref{eq:versal}.
Then we can define the {\em inward normal vector\/} $\nuin$ which points 
in the direction that the number of roots of polynomial $F(t)=0$
increasing.
For example,
on the $A_2$-singularity
$(t,x_2,\dots,x_n)\mapsto(2t^3,-3t^2,x_2,\dots,x_n)$,
$\nuin$ is defined by 
\[
  \nuin=
  \begin{cases}
   - (1,t,0,\dots,0)\quad & (t>0),\\
   \hphantom{-}
     (1,t,0,\dots,0)\quad & (t<0),
  \end{cases}
\]  
see Figure~\ref{fig:zigzag}, left.
Then, the inward direction is uniquely defined on a neighborhood of
each connected component of $S(f)$.
On the other hand, a front $f:M^n \to N^{n+1}$ is called a 
{\em Morin front\/} if there exists a $C^\infty$-map
$\hat f$ from an $(n+1)$-manifold into $N^{n+1}$
such that
\begin{itemize}
 \item $\hat f$ admits only Morin singularities,
 \item  $f$ is the restriction of $\hat f$
to the singular set $S(\hat f)$.
\end{itemize}
\noindent
Morin fronts are all A-fronts 
(see Corollary~\ref{cor:rest-front} and the appendix).
Conversely, any $A$-fronts are locally Morin fronts.
The inward direction of the  Morin front $f$ points in the
direction of the image of $\hat f$, and we get the following:
\begin{proposition}\label{prop:morin-inward}
 A Morin front admits has a unique unit normal vector on the set of
 regular ponits
 which points in the inward direction.
\end{proposition}
We fix a regular point $o\in M^n$ of $f$.
A continuous map
$
    \gamma:[0,1]\longrightarrow M^n
$
is called a {\em loop\/} at $o$ if it satisfies $\gamma(0)=\gamma(1)=o$.
We fix a Riemannian metric $g$ of $N^{n+1}$.
A loop $\gamma$ is called {\em co-orientable\/} if there exists a unit
normal vector field $\nu(t)\in T_{\hat\gamma(t)}N^{n+1}$ of 
$f$ along $\gamma$ such that $\nu(0)=\nu(1)$,
where
$\hat \gamma(t):=f\circ \gamma(t)$.
If $\gamma$ is not co-orientable, it is called
{\it non-co-orientable}.
These two properties of loops depend only on their
homotopy classes. 
We set
\[
    \rho_f(\gamma):=
\begin{cases}
0 \quad & (\text{if $\gamma$ is co-orientable}), \\
1 & (\text{if $\gamma$ is non-co-orientable}).
\end{cases}
\]
Then it induces a
representation
$
    \rho_f:\pi_1(M^n)\longrightarrow \Z_2=\Z/2\Z
$.
Next, we set
\[
    \hat\pi_1(M^n):=\ker(\rho_f)\bigl(\subset \pi_1(M^n)\bigr)
\]
which consists of homotopy classes of co-orientable
loops. 
We shall now construct the secondary homotopical representation
$\hat\rho_f:\hat\pi_1(M^n)\longrightarrow \Z$
called the {\it Maslov representation\/}
as follows:
Let $f:M^n\to N^{n+1}$ be a front and
\[
    S(f)=\bigcup_{\alpha\in \Lambda}\Gamma_{\alpha}
\]
be the decomposition of the singular set $S(f)$ by its connected
components.
Then there exists a disjoint family of open domains $O_\alpha$ in
$M^n$ such that $\Gamma_\alpha\subset O_\alpha$.
Let $\gamma:[0,1]\to M^{n}$ be a co-orientable loop at $o\not\in S(f)$.
We set
\[
    S_\gamma:=\{t\in [0,1]\,;\, \gamma(t)\in S(f)\}.
\]
For each $t\in S_{\gamma}$,
there exists an open interval $I_t(\subset (0,1))$ 
such that
$\gamma(I_t)\subset O_\alpha$
for some $\alpha\in \Lambda$.
Since $S_\gamma$ is compact, there exists
a finite number of open intervals 
$I_1=(t^{-}_1,t^{+}_{1}),\dots,
   I_m=(t^{-}_m,t^{+}_m)$
such that
$S_\gamma\subset I_1\cup \cdots \cup I_m$.
If $I_k\cap I_{k+1}$ is non-empty,
we can replace $I_k$ and $I_{k+1}$ by $I_k\cup I_{k+1}$.
After these operations, each $I_k$ is still contained in 
some $O_\alpha$, since $\{O_\alpha\}_{\alpha\in \Lambda}$ 
is disjoint.
Then we may assume that
\[
   0< t^{-}_1< t^{+}_1<
         t^{-}_2<\dots<
     t^{-}_m<
         t^{+}_m<1.
\]
We fix a unit normal vector $\nu_0\in T_{\hat \gamma(o)}N^{n+1}$ at $o$.
(If $f$ admits a globally defined smooth unit normal vector field $\nu$
on $M^n$, one of the canonical choices is $\nu_0=\nu(0)$.)
Let $\nu_j$ ($j=1,2,\dots,m$) be the 
inward unit normal vector at $t=t^{+}_j$.
Then we can take a continuous unit vector field $\hat \nu_j(t)$
($t^+_{j}\le t \le t_{j+1}^-$) along $\gamma|_{[t^+_{j},t_{j+1}^-]}$
such that 
\[
    \hat \nu_i(t^{+}_j)=\nu_j \qquad (i=0,1,2,\dots,m),
\]
where $t^{+}_0=0$ and $t^{-}_{m+1}=1$.
We set $\varepsilon_0=+1$ and
\[
    \varepsilon_j:=
      \begin{cases}
       \hphantom{-}\varepsilon_{j-1} & 
         (\mbox{%
             if $\hat \nu_{j-1}(t^{-}_{j})$ is inward}) \\
       -\varepsilon_{j-1} & 
         (\mbox{%
             if $-\hat \nu_{j-1}(t^{-}_{j})$ is inward}) \\
      \end{cases}
      \qquad (j=1,2,\dots,m+1),
\]
where $\nu_{m+1}=\nu_0$.
Then we get a sequence
\begin{equation}\label{eq:seq}
   \varepsilon_0,~\varepsilon_1,~
   \varepsilon_2,~\dots~,~\varepsilon_m,~\varepsilon_{m+1}.
\end{equation}
Since $M^n$ is orientable, $m$ is an odd integer.
By the co-orientability of $\gamma$,
$\varepsilon_0=\varepsilon_{m+1}=+1$
holds.
For simplicity, we write $\varepsilon_j=+$ (resp.~$-$) if
$\varepsilon_j=+1$ (resp. $-1$).
If there are two successive same signs $++$ or $--$, we cancel them. 
Repeating this cancellation, we get a sequence which has no successive
signs
\[
    +-\cdots+-+.
\]
We shall call this reduction the {\em normalization\/}  of
the sequence \eqref{eq:seq}.
Then the number of $-$ after the normalization
is called the {\em zig-zag number\/} of the  co-orientable 
loop $\gamma$. (The definition of zig-zag number for
surfaces in $\R^3$ is given in  \cite{LLR}. See also \cite{suy}.)
For example, we set $M=S^1$ and $N=\R^2$ and consider a front as in
Figure~\ref{fig:zigzag}, right.
\begin{figure}
\begin{center}
\begin{tabular}{cc}
\raisebox{1cm}{
 \includegraphics{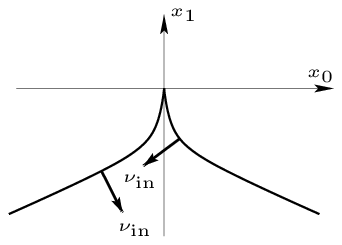}}& 
 \includegraphics{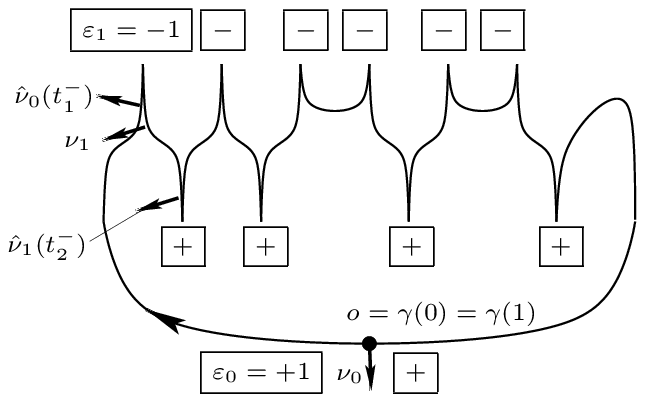} \\
 The inward normal &
 Zigzag number
\end{tabular}
\end{center}
\caption{Zigzag number of a plane curve}\label{fig:zigzag}
\end{figure}
Then this closed curve induces a
sequence
\[
    +-+-+--+--++
\]
which can be reduced to
\[
   +-+(-(+(--)+)-)-(++)\quad \mbox{that is}\quad +-+-
\]
and thus the zig-zag number in this case is $2$.
We denote by $z(\gamma)$ the zig-zag number
of the co-orientable loop $\gamma$.
(Global properties and references of zig-zag 
numbers of planar curves are given in \cite{A}.)

\begin{proposition}\label{prop:homotopy}
 The zig-zag number $z(\gamma)$
 depends only on the homotopy class of $\gamma$,
 and is stable under the deformation of an $A$-front.
\end{proposition}
\begin{proof}
 Suppose $[\gamma_0]=[\gamma_1]$ in $\pi_1(M^n)$.
 Then there exists a homotopy $\Gamma:[0,1]\times [0,1]\to M^n$
 such that $\Gamma(t,0)=\gamma_0(t)$
 and $\Gamma(t,1)=\gamma_1(t)$. We set
 \[
      S_\Gamma:=\{(t,s)\in [0,1]\times[0,1]
      \,;\,
      \Gamma(t,s)\in S(f)\}.
 \]
 For each $(t,s)\in S_\Gamma$,
 there exists an open neighborhood
 $W_{t,s}\subset [0,1]\times [0,1]$ of $(t,s)$ such that
 $\Gamma(W_{t,s})\subset O_\lambda$ for some $\lambda\in \Lambda$.
 Since $[0,1]\times [0,1]$ is compact,
 it is covered by just finitely many
 such open subsets.
 Then the standard homotopy argument yields
 the homotopy invariance $z(\gamma_0)=z(\gamma_1)$. 
\end{proof}

A closed $C^1$-regular curve $\gamma:[0,1]\to M^n$ 
starting at $o$
is called a {\em null loop\/}
if $\gamma'(t)$ ($t\in [0,1]$) belongs to the kernel of $df$
whenever $\gamma(t)\in S(f)$.
One can easily show that, for each homotopy class of the fundamental
group $\pi_1(M^n)$,
there exists a null $C^1$-loop which represents the homotopy class.
Moreover, we may assume that $\gamma(t)$ passes through only finitely
many $A_2$-singular points.
Moreover, we may assume that $\gamma(t)$ passes through only finitely
many $A_2$-singular points.

Suppose that $\gamma(t)$ is a co-orientable null loop.
Then there exists a continuous unit normal vector field 
$\nu(t)$ along $\gamma(t)$ such that $\nu(0)=\nu_0$.
Since $\gamma$ is co-orientable, $\nu(1)=\nu_0$ holds.
As in \cite{suy},
we define a continuous map of $[0,1]\setminus S_\gamma$
into  $P^1(\R)$ by
\[
    C_{\gamma}:
    t\longmapsto
    [g(\hat\gamma'(t),\hat\gamma'(t)):g(\nu'_\gamma(t),\hat\gamma'(t))],
\]
which can be extended continuously across $S_\gamma$.
We call this map $C_\gamma:[0,1]\to P^1(\R)$
{\em the normal curvature map}.
In fact, when $\gamma(t)$ is a regular point of $f$,
$-g(\nu'_\gamma,\hat\gamma')/g(\hat\gamma',\hat\gamma')$
is exactly the normal curvature of $f$ along $\gamma$.
This is independent of the orientation of $\gamma$,
but its sign depends on the initial choice of the unit normal vector
$\pm\nu_0$ at $o$.
We denote the rotation index of the map $C_\gamma$
by $\mu_f(\gamma)$.
The absolute value of $\mu_f(\gamma)$ is called the {\em Maslov index\/}
of the co-orientable null loop $\gamma$.
Then by the same argument as in \cite{suy},
we get the following
\begin{proposition}
 The Maslov index $|\mu_f(\gamma)|$ 
 coincides with the zig-zag number $z(\gamma)$.
 In particular, we get a {\rm (}Maslov{\rm )} representation
 \[
    \hat \rho_f:
     \hat \pi_1(M^n) \ni [\gamma]\longmapsto \mu_f(\gamma)
         \in \Z.
 \]
\end{proposition}
Since the inward direction is defined globally for a Morin front,
we have the following:
\begin{corollary}
   The Maslov representations for Morin fronts are trivial.
\end{corollary}

\begin{remark}
 Suppose that $M^n$ and $N^{n+1}$ are both oriented and $f$ is
 co-orientable.
 Then, we can take a globally defined unit normal 
 vector field $\nu$ along $f$ which gives a 
 direct interpretation of the sign $\varepsilon_j$:
 We denote by $O_\alpha^+$ the component of
 $O_{\alpha}\setminus\Gamma_{\alpha}$
 where $\nu$ is compatible with the orientation of $M^n$.
 Then each component $\Gamma_{\alpha}$ of $S(f)$ is called 
 {\em zig\/} if $\nu$ coincides with $\nuin$ on $O_{\alpha}^+$.
 Some global properties of zig-zag numbers are given in
\cite{LLR}.
\end{remark}

\appendix
\section{Criteria of Morin type singularities}
\label{sec:morin}

The {\em $A_k$-Morin singularities\/} are map germs which are
 $\K$-right-left equivalent to
 \begin{equation}\label{eq:morink}
      f(z_1,\ldots,z_n)
       =
      (z_1z_n+\dots+z_{k-1}z_n^{k-1}+z_n^{k+1},z_1,\dots,z_{n-1})
      \qquad (k\leq n)
 \end{equation}
at $\vect{0}$ (\cite{morin}, see also \cite{gg}).
The $A_0$-Morin singularity means
a regular point.
The following assertion holds:

\begin{theorem}\label{thm:morin-cri}
 Assume that $k\leq n$.
 Let $U$ be a domain in $\K^n$, and $f\colon{}U\to \K^n$ a
 $\K$-differentiable map and $p\in S(f)$.
 We assume that $p$ is a corank one singularity.
 Then $f$ at $p$ is $\K$-right-left equivalent to $A_k$-Morin
 singularity if and only if
 \begin{align}
    \label{eq:morin-cond1}
    &\lambda=\lambda'=\cdots=\lambda^{(k-1)}=0,\ \lambda^{(k)}\ne0
    \ \text{at\ }p \quad\text{and}\\
    \label{eq:morin-cond2}
    &\Lambda:=(\lambda,\lambda',\ldots,\lambda^{(k-1)})
    :U\to\K^{k}\quad\text{is non-singular at $p$.}
 \end{align}
 Here, 
 $\lambda=\det(f_{x_1},\dots,f_{x_n})$,
 $(x_1,\dots,x_n)$ is the canonical coordinate system on $U$,
 $\lambda'=\tilde\eta\lambda(=\lambda^{(1)})$,
 $\lambda^{(i)}=\tilde\eta\lambda^{(i-1)}$
 and
 $\tilde\eta$ is the extended null vector field of $f$.
\end{theorem}

\begin{proof}
 Since the conditions \eqref{eq:morin-cond1} and \eqref{eq:morin-cond2}
 do not change when changing
 $\lambda$ to $\psi\lambda$, where $\psi(p)\ne0$,
 the conditions \eqref{eq:morin-cond1} and \eqref{eq:morin-cond2}
 are independent on the choice of local coordinate systems
 on the source and target.
 Suppose $f$ has the form 
\begin{equation}\label{eq:form}
    f(z_1,\ldots,z_n)=\bigl(\mu(z_1,\dots,z_n),z_1,\dots,z_{n-1}\bigr), 
\end{equation}
such that
 \begin{itemize}
  \item $\mu'=\cdots=\mu^{(k)}=0$, $\mu^{(k+1)}\ne0$ at $p$
	$('=\partial/\partial z_n)$, and
  \item the map germ of $(\mu',\ldots,\mu^{(k)})$ 
	is nonsingular at $p$.
 \end{itemize}
 Then $f$ is an $A_k$-Morin singularity (see \cite[page
 177]{gg}).
 Though the proof of this fact in \cite{gg}
 is given for $\K=\R$, the argument works also for the case of
 $\K=\C$.
 By \eqref{eq:form}, we have
 $\lambda=\mu'$, and $\eta=\partial_{z_n}$ holds.
 This implies the theorem.
\end{proof}
The same type of assertion as in 
Corollaries~\ref{cor:submanifold1} and \ref{cor:submanifold2}
hold for $A_k$-Morin singularities.
\begin{acknowledgements}
 The authors thank Shyuichi Izumiya and Goo Ishikawa 
 for fruitful discussions and valuable comments.
\end{acknowledgements}

\end{document}